\documentclass[12pt]{article}

\usepackage{amsthm,tikz,amssymb, latexsym, amsmath, amscd, amsfonts, array, graphicx, lmodern, slantsc,enumerate,stmaryrd,rotating}

\usepackage[english]{babel}        
\usepackage[all]{xy}
\CompileMatrices

\usepackage{color}
\def\red{}

\def\pt{\hspace{.2mm}{\cdot}\hspace{.3mm}}

\def\pp{\protect\operatorname{PP}}
\def\card{\protect\operatorname{card}}
\def\hdim{\protect\operatorname{hdim}}
\def\ppn{\protect\operatorname{PP}_n}
\def\TC{\protect\operatorname{TC}}

\def\F{\protect\operatorname{Conf}}
\def\confkRn{\ensuremath{\F_k(\mathbb{R},n)}}
\def\UF{\protect\operatorname{UConf}}
\def\cat{\protect\operatorname{cat}}
\def\secat{\protect\operatorname{secat}}
\def\zcl{\protect\operatorname{zcl}}
\def\cl{\protect\operatorname{cl}}
\def\gdim{\protect\operatorname{gdim}}
\def\cdim{\protect\operatorname{cdim}}
\def\conn{\protect\operatorname{conn}}

\newtheorem{proposition}{Proposition}[section]
\newtheorem{corollary}[proposition]{Corollary}
\newtheorem{definition}[proposition]{Definition}
\newtheorem{theorem}[proposition]{Theorem}
\newtheorem{remark}[proposition]{Remark}
\newtheorem{example}[proposition]{Example}
\newtheorem{lemma}[proposition]{Lemma}

\begin{document}

\title{Linear motion planning with controlled collisions and pure planar braids}
\author{Jes\'us Gonz\'alez, Jos\'e Luis Le\'on-Medina \\ and Christopher Roque}
\date{\empty}

\maketitle
\tableofcontents

\begin{abstract}
We compute the Lusternik-Schnirelmann category (LS-$\cat$) and the higher topological complexity ($\TC_s$, $s\geq2$) of the ``no-$k$-equal'' configuration space $\F_k(\mathbb{R},n)$. This yields (with $k=3$) the LS-cat and the higher topological complexity of Khovanov's group $\ppn$ of pure planar braids on $n$ strands, which is an $\mathbb{R}$-analogue of Artin's classical pure braid group on $n$ strands \cite{MR1386845}. Our methods can be used to describe optimal motion planners for $\ppn$ provided $n$ is small.
\end{abstract}



\section{Introduction}
For a topological space $X$ and a positive integer $n$, the configuration spaces $\F(X,n)=\{(x_1,\ldots,x_n)\in X^n\colon x_i\neq x_j \mbox{ for } i\neq j \}$, of $n$ ordered points in $X$, and $\UF(X,n)$, the orbit space of $\F(X,n)$ by the canonical action of the $n$-th permutation group, are central objects of study in pure and applied mathematics. The case $X=\mathbb{C}$ is historically and theoretically important: both $\F(\mathbb{C},n)$ and $\UF(\mathbb{C},n)$ are Eilenberg-MacLane spaces of respective types $(\text{P}_n,1)$ and $(\text{B}_n,1)$. Here $\text{B}_n$ stands for Artin's classical braid group on $n$ strands, and $\text{P}_n$ denotes the corresponding subgroup of pure braids.

Having contractible path components, $\F(\mathbb{R},n)$ and $\UF(\mathbb{R},n)$ are topologically uninteresting. A meaningful and rich $\mathbb{R}$-analogue of $\mathbb{C}$-based configuration spaces arises when the actual definition of a configuration space is relaxed.

For $X$ and $n$ as above, and for an integer $k\geq2$, the ``no-$k$-equal'' (ordered) configuration space $\F_k(X,n)$ is the subspace of the product $X^n$ consisting of the $n$-tuples $(x_1,\ldots,x_n)$ for which no set $\{x_{i_1},\ldots,x_{i_k}\}$, with $i_j\neq i_\ell$ for $j\neq\ell$, is a singleton. The corresponding unordered analogue U$\F_k(X,n)$ is the orbit space of $\F_k(X,n)$ by the canonical action of the $n$-th permutation group. As shown in~\cite{MR3572355}, the homotopy properties of $\F_k(X,n)$ ($\UF_k(X,n)$) interpolate between those of the usual configuration space $\F(X,n)=\F_2(X,n)$ (U$\F(X,n)=\mbox{U}\!\F_2(X,n)$), and those of the cartesian (symmetric) $n$-th power $X^n=\F_k(X,n)$ (SP${}^nX=\mbox{U}\!\F_k(X,n)$), for $k>n$. Moreover, as discussed in~\cite{MR3426383}, no-$k$-equal configuration spaces play a subtle role in the study of the limit of Goodwillie's tower of a space of no $k$-self-intersecting immersions.

For the particular case $k=3$, $\F_3(\mathbb{R},n)$ gives the desired $\mathbb{R}$-analogue of the classical Artin pure braid group. In \cite{MR1386845}, Khovanov introduces $\ppn$ which stands for the group of planar pure braids on $n$ strands, also called pure twin group, and proves that $\F_3(\mathbb{R},n)$ is an aspherical space which classifies $\ppn$-principal bundles. No-$k$-equal configuration spaces on the real line were first considered in~\cite{BLY}, where methods for estimating the size and depth of decision trees are applied to the analysis of the complexity of the problem of determining whether, for given $n$ real numbers, some $k$ of them are equal.

A central goal of this paper is the computation of Farber's topological complexity (TC) of $\F_k(\mathbb{R},n)$ for $k\geq3$. In the process, we compute the Lusternik-Schnirelmann category (cat) and all the higher topological complexities ($\TC_s$, $s\ge2$) of $\F_k(\mathbb{R},n)$.

\begin{theorem}\label{teoremaprincipal}
The Lusternik-Schnirelmann category and the topological complexity of $\F_k(\mathbb{R},n)$ are given by
\begin{align}
\cat\left(\F_k(\mathbb{R},n)\right)={}& \lfloor n/k\rfloor, \mbox{ the integral part of $n/k$, and}\nonumber\\
 \TC(\confkRn)={}&
\begin{cases}
0, & n<k;\\
1, & n=k \mbox{ with $k$ odd}\hspace{.3mm};\\2, & n=k \mbox{ with $k$ even}\hspace{.3mm};\\
2\lfloor n/k\rfloor, & n>k.
\end{cases}\label{respuestacompleta}
\end{align}
\end{theorem}

See Corollary~\ref{lassecuenciales} for the corresponding description of all the higher topological complexities of $\confkRn$. Note that $\TC(\F_k(\mathbb{R},n))=2\lfloor n/k \rfloor$ unless $n=k=2\ell+1$ for some $\ell>0$.

It is worth highlighting a couple of partial similarities between Theorem~\ref{teoremaprincipal} and the topological complexity of the classical configuration spaces $\F(\mathbb{R}^d,n)$ described in~\cite{MR2470845}:
\begin{equation}
\TC(\F(\mathbb{R}^d,n))=\begin{cases} 2n-3, & \mbox{$d$ even;}\\2n-2, & \mbox{$d$ odd.}
\end{cases}\label{markmichael}
\end{equation}
Firstly, both~(\ref{respuestacompleta}) and~(\ref{markmichael}) are linear functions on $n$, of slope 2 in the case of~(\ref{markmichael}), and slope roughly $2/k$ ($1/k$ if $n=k=2\ell+1$) in the case of~(\ref{respuestacompleta}). Further, just as in~(\ref{markmichael}), (\ref{respuestacompleta}) is at most one from maximal possible; (\ref{markmichael}) is precisely one less than maximal possible for $d$ even, while~(\ref{respuestacompleta}) is so only for $n=k$, an odd number.

Since $\TC(X)$ is a homotopy invariant of $X$, the topological complexity of a group $G$ can be defined as that of any of its classifying spaces, just as in the case of the Lusternik-Schnirelman category $\cat(G)$. In the short but influential paper~\cite{MR0085510}, Eilenberg and Ganea laid the grounds for establishing the fact that $\cat(G)$ agrees with the projective dimension of the trivial $\mathbb{Z}[G]$-module $\mathbb{Z}$. On the other hand, a description of $\TC(G)$ depending solely on the algebraic properties of $G$ is an open problem which has captured much of the current attention of the experts in the field.

\begin{corollary}\label{tprincipal}
The category and the topological complexity of $\ppn$ are given by
\begin{align*}
\cat(\ppn)&{}=\lfloor n/3\rfloor,\\
\TC(\ppn)&{}=
\begin{cases}
0, & n<3;\\
1, & n=3;\\
2\lfloor n/3\rfloor, & n>3.
\end{cases}
\end{align*}
\end{corollary}

\begin{remark}{\em
J.~Mostovoy pointed out to the authors that the cartesian product of $\lfloor n/3\rfloor$ copies of $\pp_3$ sits inside $\ppn$ (by cabling sets of 3-strands). In particular, $\ppn$ is hyperbolic only for $n=3,4,5$. In this respect, it is relevant to observe that, while the main result in~\cite{FMHyperbolic} asserts that the topological complexity of a hyperbolic group $\pi$ must be $\cdim(\pi\times\pi)-\delta_\pi$ with $\delta_\pi\in\{0,1\}$, $\pp_3$ seems to be the only known hyperbolic group $\pi$ with $\delta_\pi=1$.
}\end{remark}

Cases with $n\leq5$ in Corollary~\ref{tprincipal} are recovered in Section~\ref{seccionhn} with short proofs of the facts that $\pp_1$ and $\pp_2$ are trivial groups, whereas $\pp_3$, $\pp_4$ and $\pp_5$ are free groups of respective ranks $1$, $7$, $31$.  (The assertion for $\pp_5$ appears as Conjecture~3.5 in~\cite{Barda}.) The fact that $\ppn$ is free for $3\le n\le 5$ has an interesting reminiscence for $n=6$. A direct computation (verifiable using the computational algebraic system GAP) using the Reidemeister-Schreier process reveals a group isomorphism $\pp_6\cong H_6\ast F$, where $F$ is a free group of rank at least $45$. Details of such a fact, as well as potential extensions for groups $\ppn$ with $n\geq6$, are the topic of the forthcoming paper~\cite{MR}. Here we remark that, in any decomposition $\pp_n\cong H_n\ast F$ with $F$ free, the $\cat$ and $\TC$ values of $H_n$ are forced to agree with those of $\pp_n$.
\begin{corollary} Assume a group isomorphism $\ppn\cong H_n\ast F$ holds for $n\ge6$ with $F$ a free group. Then
$\cat(H_n)=\lfloor n/3\rfloor$ and $\TC(H_n)=2\lfloor n/3\rfloor$.
\end{corollary}
\begin{proof}
This is an immediate consequence of Corollary~\ref{tprincipal} and the formulae
\begin{align*}
\cat(G_1\ast G_2)&{}={}\max\{\cat(G_1),\cat(G_2)\},\\
\TC(G_1\ast G_s)&{}=\max\{\TC(G_1),\TC(G_2),\cat(G_1\times G_2)\}
\end{align*}
for the free product $G_1\ast G_2$ of arbitrary groups $G_1$ and $G_2$ (the $\TC$ formula has recently been proved in~\cite{dransady}).
\end{proof}

Theorem~\ref{teoremaprincipal} has potential applications to current technological developments. For instance, $\confkRn$ is the state space of a system consisting of $n$ distinguishable points moving on an interval, and subject to the restriction that $k$-multiple collisions are forbidden. For practical applications it is convenient to replace points by intervals of a fixed (suitably small) radius, changing the no-$k$-multiple-collision condition by the requirement that no $k$ intervals have a common overlapping. Indeed, it is known (see~\cite{MR3426383}) that the configuration space based on intervals is homotopic to the one based on points. In this context, if the moving objects are equipped with communication sensors, and the radius of the intervals are thought of as the communication range of each of the moving objects, then the no-$k$-multiple-collision condition corresponds to the requirement that at most $k-1$ vehicles moving on a highway can communicate at any given time. 

With an eye on further potential applications, no-$k$-equal configuration spaces are generalized in the short \red{final} Section~\ref{sectioncontrolled}, where we introduce configuration spaces $\F_K(X,n)$ with collisions controlled by a simplicial complex $K$. A driving motivation (that arose from a lecture of~\cite[Section~2.4]{MR2605308}) is that such spaces (with $X=\Gamma$ a graph) would seem to be a natural space of states for problems in digital microfluidics (see~\cite{digit,wet}). In such processes, manipulation of droplets embedded on an inert oil suspension is performed by suitable application of currents through a grid of wires (the graph $\Gamma$) in order to propel droplets through the wires (due to dynamic surface tension effects). In such a setting, motion planning with controlled collisions (encoded by the complex $K$) corresponds to specific mixing process instructions: droplets of various chemical or biological agents would be positioned, mixed, split, and directed to final outputs, all in parallel ---an efficient ``lab on a chip''.

It should also be remarked that configuration spaces $\F_K(X,n)$ with collisions controlled by a simplicial complex $K$ are also interesting outside applications. We thank Victor Turchin for bringing to our attention that these spaces (with $X=\mathbb{R}$) were used in~\cite{dobrinskayaloops} to study the homology of the loop space on a polyhedral product $(X_1,\ldots,X_m)^K$ when each space $X_i$ is simply connected.

\section{Preliminaries}
\subsection{\red{LS} category and topological complexity}\label{subsectionTC}

\red{For a space $X$,} the Lusternick-Schnirelmann category, \red{cat$(X)$}, and the topological complexity, \red{$\TC(X)$, both homotopy invariants of $X$, are} special cases of the notion of sectional category (or Schwarz genus) of a fibration. Recall that the (reduced) sectional category of a fibration $p:E\to B$, $\secat(p)$, is defined as the \red{smallest non-negative integer} $k$ so that there exists an open covering of the base $B = \red{U_0 \cup U_1 \cup \cdots \cup U_k}$ such that the fibration $p$ admits a \red{continuous} section on each $U_i$, see \cite{Sch58}\footnote{\red{Schwarz' original (unreduced) definition is recovered as $\text{genus}(p)=\secat(p)+1$.}}. As \red{a} special \red{case}, we obtain the (reduced) Lusternick-Schnirelmann category of a space $X$, $\cat(X)$, defined as the sectional category of the fibration $e_{1}\colon P_0(X)\to X$, where $P_0(X)$ is the space of based paths on $X$ and $e_1$ is the evaluation map given by $e_1(\gamma)=\gamma(1)$. On the other hand, the (reduced) topological complexity of a space $X$, $\TC(X)$, is defined as the sectional category of the fibration $e_{0,1}\colon P(X)\to X\times X$, where $P(X)$ is the space of free paths on $X$ and $e_{0,1}$ is the double evaluation  map given by $e_{0,1}(\gamma)=(\gamma(0),\gamma(1))$. The open\footnote{For practical purposes, the openness condition on local domains can be replaced (without altering the resulting numerical value of $\TC(X)$) by the requirement that local domains are pairwise disjoint Euclidean neighborhood retracts (ENR).} sets $U_i$ covering $X\times X$ so that $e_{0,1}$ admits a continuous section on each $U_i$ are called local domains, and the corresponding local sections are called local rules. The system of local domains and local rules is called a motion planner for~$X$. A motion planner is said to be optimal if it has $\TC(\red{X})$ local rules. 
As explained by Farber in his seminal work~\cite{Far,MR2074919}, this concept gives a homotopical framework for studying the motion planning problem in robotics. Indeed, $\TC(X)$ gives a measure of the complexity of motion-planning an autonomous system with state-space $X$ and which should perform robustly within a noisy environment.

Most of the existing methods to estimate the topological complexity of a given space are cohomological in nature and are based on some form of obstruction theory. One of the most (simple and) successful such methods is: 

\begin{proposition}\label{ulbTCn}
Let $X$ be a $c$-connected space $X$ having the homotopy type of a CW complex, then 
\red{$$\cl(X)\leq\cat(X)\leq \hdim(X)/(c+1)
\quad\mbox{and}\quad
\zcl(X)\leq\TC(X)\leq 2 \cat(X).$$}
\end{proposition}

The notation $\hdim(X)$ stands for the (cellular) homotopy dimension of $X$, i.e.~the minimal dimension of CW complexes having the homotopy type of $X$. On the other hand, the \red{cup-length of $X$, $\cl(X)$, and the} zero-divisor cup-length of $X$, $\zcl(X)$, \red{are} defined in purely cohomological terms. \red{The former one is the largest non-negative integer $\ell$ such that there are coefficients systems $A_1,\ldots,A_\ell$ over $X$ and corresponding positive-dimensional classes $c_j\in H^*(X;A_{j})$ so that the product $c_1\cdots c_\ell\in H^*(X;\bigotimes_i A_i)$ is non-zero. Likewise, $\zcl(X)$} is the largest non-negative integer $\ell$ such that there are coefficients systems $A_1,\ldots,A_\ell$ over $X\times X$ and corresponding classes $z_j\in H^*(X\times X;A_{j})$, each with trivial restriction under the diagonal inclusion $\Delta\colon X \hookrightarrow X\times X$, and so that the product $z_1\cdots z_\ell\in H^*(X\times X;\bigotimes_i A_i)$ is non-zero. Each such class $z_i$ is called a zero-divisor for $X$. Throughout this work, we will only be concerned with simple coefficients in $\mathbb{Z}_2$, and will omit reference of coefficients in writing a cohomology group $H^*(X)$. In these terms, $\Delta^*\colon H^*(X\times X)=H^*(X)\otimes H^*(X)\to H^*(X)$ is given by cup-multiplication, which explains the name ``zero-divisors''.

All definitions and results reviewed in this subsection have corresponding analogues for Rudyak's higher topological complexity, see~\cite{bgrt,Ru10} for details.

\subsection{Preorders and the cohomology ring of \confkRn}\label{sectioncohomology}
We recall the description of the cohomology ring\footnote{Recall we only take mod 2 coefficients.} $H^*(\confkRn)$ ---see~\cite{Bary,MR3426383}. A binary relation $R$ on a set $S$ is any subset of the cartesian product $S\times S$. As usual, we write $xRy$ as a substitute for $(x,y)\in R$. A preorder is a binary relation $\preceq$ on $S$ which is reflexive ($x\preceq x, \forall x\in {S}$) and transitive ($x\preceq y\preceq z\Rightarrow x\preceq z,\forall x,y,z\in{S}$). For instance, the diagonal $\Delta_S=\{(x,x)\colon x\in S\}$ and the entire cartesian product $S\times S$ are preorders which are called empty and full, respectively. Whenever the preorder is understood, a situation where $x\preceq y$ and $y\preceq x$ is denoted by $x\approx y$. Thus, a partial order is a preorder where $x\approx y$ occurs only with $x=y$. We write $x\prec y$ when both $x\preceq y$ and $x\neq y$ hold.

The transitive closure of a binary relation $R$ on $S$ is the smallest transitive binary relation on $S$ containing $R$. In particular, the transitive closure of (the union of) two preorders is automatically a preorder. This yields a commutative and associative binary operation $\circ\colon\mathcal{P}(S)\times\mathcal{P}(S)\to\mathcal{P}(S)$ on the set $\mathcal{P}(S)$ of preorders on $S$ having the empty preorder as a two-sided neutral element.

Fix positive integers $n$ and $k$ with $3\le k\leq n$. Baryshnikov describes the cohomology ring $H^*(\confkRn)$ in terms of what he calls \emph{string} preorders, i.e., preorders which are \emph{almost} determined by a ``height'' function. Explicitly, a preorder $\preceq$ on $[n]$ is string if there is a preorder-preserving map $h\colon [n]\to \mathbb{R}$ (where $\mathbb{R}$ is equipped with the standard order) satisfying $x\prec y$ whenever $h(x) < h(y)$, and in such a way that the restriction of $\preceq$ to each ``level'' set $h^{-1}(r)$ ($r\in\mathbb{R}$) is either the empty preorder or the full preorder (the height function would fully recover the string preorder if the former would remember which level sets are empty and which are full). Thus, a string preorder $\preceq$ can be spelled out through the ordered list (or string) of non-empty level sets of a corresponding height function for~$\preceq$, where the list is ordered increasingly\footnote{In~\cite{Bary}, level sets are ordered decreasingly from left to right; this difference is immaterial.} from left to right according to the height values, and enclosing each level subset $I\subseteq[n]$ within either [ ]-brackets, if the restriction of $\preceq$ to $I$ is full,  or ()-brackets, if the restriction of $\preceq$ to $I$ is empty. By convention, a level set with a single element has to be enclosed within ()-brackets. 

A string preorder is said to be:
\begin{itemize}
\item[(a)] \emph{elementary}, if it has the form $(I)[J](K)$ with $\card(J)=k-1$.
\item[(b)] \emph{admissible}, if it has the form $(I_0)[J_1](I_1)[J_2]\cdots[J_d](I_d)$ with $\card(J_i)=k-1$ for all $i=1,\ldots d$. In such a case, the admissible string preorder is said to have \emph{dimension~$(k-2)d$}. Elementary string preorders are thus admissible and have dimension~$k-2$.
\item[(c)] \emph{basic}, if it is specified by a string $(I_0)[J_1](I_1)[J_2]\cdots[J_d](I_d)$ satisfying $\card(J_i)=k-1$ and $\max(J_i\cup I_i)\in I_i$, for all $i=1,\ldots, d$ (the maximal element of $J_i\cup I_i$ is taken with respect to the standard order of integers).
\end{itemize}

\begin{remark}\label{pdd2}{\em
An admissible (basic) preorder $(I_0)[J_1](I_1)[J_2]\cdots[J_d](I_d)$ of dimension $(k-2)d$
factors as $\varepsilon_1\circ\cdots\circ\varepsilon_d$, where 
$$\varepsilon_i=\left(\rule{0mm}{4mm}I_0\cup J_1\cup I_1\cup\cdots\cup J_{i-1}\cup I_{i-1}\right)\left[\rule{0mm}{4mm}J_i\right]\left(\rule{0mm}{4mm}I_i\cup J_{i+1}\cup I_{i+1}\cup\cdots\cup J_d\cup I_d\right)$$
is an elementary (basic) preorder of dimension $k-2$. As a partial converse, note that, for string preorders $(I)[J](K)$ and $(I')[J'](K')$ (possibly non-elementary), the condition $I\cup J\subseteq I'$ implies the equality
\begin{equation}\label{pdeee}
\left(\rule{0mm}{4mm}I\right)\left[\rule{0mm}{4mm}J\right]\left(\rule{0mm}{4mm}K\right)\circ\left(\rule{0mm}{4mm}I'\right)\left[\rule{0mm}{4mm}J'\right]\left(\rule{0mm}{4mm}K'\right)=\left(\rule{0mm}{4mm}I\right)\left[\rule{0mm}{4mm}J\right]\left(\rule{0mm}{4mm}K\cap I'\right)\left[\rule{0mm}{4mm}J'\right](K').
\end{equation}
(If $K\cap I'=\varnothing$, the ()-level set $K\cap I'$ must be suppressed from the right term in~(\ref{pdeee}).) Likewise, the condition $I'\cup J'\subseteq I$ implies the equality
\begin{equation}\label{pdeeebis}
\left(\rule{0mm}{4mm}I\right)\left[\rule{0mm}{4mm}J\right]\left(\rule{0mm}{4mm}K\right)\circ\left(\rule{0mm}{4mm}I'\right)\left[\rule{0mm}{4mm}J'\right]\left(\rule{0mm}{4mm}K'\right)=\left(\rule{0mm}{4mm}I'\right)\left[\rule{0mm}{4mm}J'\right]\left(\rule{0mm}{4mm}K'\cap I\right)\left[\rule{0mm}{4mm}J\right](K).
\end{equation}
(Correspondingly, if $K'\cap I=\varnothing$, the ()-level set $K'\cap I$ must be suppressed from the right term of~(\ref{pdeeebis}).) The apparent symmetry on the right of~(\ref{pdeee}) and~(\ref{pdeeebis}) corresponds to the fact that the condition $I\cup J\subseteq I'$ ($I'\cup J'\subseteq I$) is equivalent, by complementing, to the condition $J'\cup K'\subseteq K$ ($J\cup K\subseteq K'$).
}\end{remark}

The fact that the products in~(\ref{pdeee}) and~(\ref{pdeeebis}) are string does not depend on the assumed inclusions $I\cup J\subseteq I'$ and $I'\cup J'\subseteq I$. Such a property is not explicitly mentioned (but is certainly used) in the original works~\cite{Bary,MR3426383}. We include proof details for completeness.

\begin{lemma}\label{d1sequeda}
The product of two string preorders $(I)[J](K)$ and $(I')[J'](K')$ is string. In particular, if neither the inclusion $I\cup J\subseteq I'$ nor the inclusion $I'\cup J'\subseteq I$ hold, then
\begin{equation}\label{chistoso}
\left(\rule{0mm}{4mm}I\right)\left[\rule{0mm}{4mm}J\right]\left(\rule{0mm}{4mm}K\right)\circ\left(\rule{0mm}{4mm}I'\right)\left[\rule{0mm}{4mm}J'\right]\left(\rule{0mm}{4mm}K'\right)=\left(\rule{0mm}{4mm}I\cap I'\right)\left[\rule{0mm}{4mm}J\cup J'\cup\left(I\cap K'\right)\cup\left(I'\cap K\right)\right]\left(\rule{0mm}{4mm}K\cap K'\right).
\end{equation}
\end{lemma}
\begin{proof}
Let $\preceq$ stand for the product preorder $(I)[J](K)\circ(I')[J'](K')$, and for subsets $A$ and $B$ of $[n]$ write $A\preceq B$ ($A\approx B$) whenever $a\preceq b$ ($a\approx b$) for all  $(a,b)\in A\times B$. For instance, $I\preceq J\preceq K$ as well as $I'\preceq J'\preceq K'$. Since $$[n]=\left(I\cap I'\right)\coprod\left(\rule{0mm}{4mm}J\cup J'\cup (I\cap K')\cup(I'\cap K)\right)\coprod\left(K\cap K'\right)$$ is clearly a partition, it \red{suffices} to show that $J\approx J'\approx(I\cap K')\approx(I'\cap K)$.

Pick $x\in (I\cup J)\setminus I'$ and $x'\in (I'\cup J')\setminus I$. For any $(j,j')\in J\times J'$, we have
\begin{itemize}
\item ($x'\not\in I\Rightarrow x'\in J\cup K \Rightarrow j\preceq x'$) and ($x'\in I'\cup J'\Rightarrow x'\preceq j'$), thus $j\preceq j'$;
\item ($x\not\in I'\Rightarrow x\in J'\cup K' \Rightarrow j'\preceq x$) and ($x\in I\cup J\Rightarrow x\preceq j$), thus $j'\preceq j$.
\end{itemize}
Therefore $J\approx J'$. The result follows from ($I\cap K'\preceq J\approx J'\preceq K'\Rightarrow I\cap K'\approx J\approx J'$) and ($I'\cap K\preceq J'\approx J\preceq K\Rightarrow I'\cap K\approx J'\approx J$).
\end{proof}

\begin{corollary}\label{conddelem}
Assume $(I)[J](K)$ and $(I')[J'](K')$ are elementary preorders with $I\cup J\nsubseteq I'$ and $I'\cup J'\nsubseteq I$. Then the product in~(\ref{chistoso}):
\begin{itemize}
\item has the form $(I'')[J''](K'')$ with $\card(J'')\ge k-1$.
\item  is elementary if and only if the preorders $(I)[J](K)$ and $(I')[J'](K')$ agree.
\end{itemize}
\end{corollary}
\begin{proof}
Note that~(\ref{chistoso}) is elementary if and only if $J=J'$ and $I\cap K'=\varnothing=I'\cap K$.
In such a case:
\begin{itemize}
\item $I\cap K'=\varnothing\Rightarrow I\subseteq I'\cup J'=I'\cup J\Rightarrow I\subseteq I'$;
\item $I'\cap K=\varnothing\Rightarrow I'\subseteq I\cup J=I\cup J'\Rightarrow I'\subseteq I$;
\item $I\cap K'=\varnothing\Rightarrow K'\subseteq J\cup K=J'\cup K\Rightarrow K'\subseteq K$;
\item $I'\cap K=\varnothing\Rightarrow K\subseteq J'\cup K'=J\cup K'\Rightarrow K\subseteq K'$.
\end{itemize}
So in fact $I=I'$ and $K=K'$. 
\end{proof}

We are now ready to state Baryshnikov's description of the ring $H^*(\confkRn)$. Recall we are assuming $\mathbb{Z}_2$ coefficients.

\begin{theorem}[{Baryshnikov~\cite[Theorem~1]{Bary}, Dobrinskaya-Turchin~\cite[Section~4]{MR3426383}}]\label{cohbar}
For $k\ge3$, the cohomology ring $H^*(\confkRn)$ is isomorphic to the (anti)commutative free exterior algebra generated in dimension $k-2$ by the elementary preorders subject to the following relations:
\begin{enumerate}
\item $\sum_{a\in A} \left(A\setminus\{a\}\right)\left[\rule{0mm}{4mm}\{a\}\cup B\right]\left(C\right) = \sum_{c\in C} \left(A\right)\left[\rule{0mm}{4mm}B\cup\{c\}\right]\left(C\setminus\{c\}\right),$ whenever $[n]$ can be written as a disjoint union $[n]=A\coprod B\coprod C$ with $\card(B)=k-2$.
\item $(I)[J](K)\cdot(I')[J'](K')=0$, for elementary preorders $(I)[J](K)$ and $(I')[J'](K')$ whose transitive closure $(I)[J](K)\circ(I')[J'](K')$ has a \text{\em [ ]}-level set of cardinality larger than $k-1$.
\end{enumerate}
\end{theorem}

\begin{remark}\label{cuandocero}{\em
Since $H^*(\confkRn)$ is a quotient of an exterior algebra, \red{Remark}~\ref{pdd2}, Lemma~\ref{d1sequeda} and Corollary~\ref{conddelem} imply that a (cup) product $(I)[J](K)\cdot(I')[J'](K')$ of two elementary preorders of dimension 1, $(I)[J](K)$ and $(I')[J'](K')$, can (potentially) be non-zero only when the (transitive-closure) product $(I)[J](K)\circ(I')[J'](K')$ is of dimension 2. Further, the latter condition holds precisely when one (and necessarily only one) of the inclusions $I\cup J\subseteq I'$ and $I'\cup J'\subseteq I$ holds, in which case the (transitive closure) product $(I)[J](K)\circ(I')[J'](K')$ is given by~(\ref{pdeee}) and~(\ref{pdeeebis}), respectively. 
}\end{remark}

Going one step further, Baryshnikov shows that the difference between cup products and transitive-closure products
 can safely be neglected:

\begin{theorem}[{Baryshnikov~\cite[Theorem~2]{Bary}, Dobrinskaya-Turchin~\cite[Section~4]{MR3426383}}]\label{basbar}
Additively, $H^*(\confkRn)$ is free with (graded) basis given by the cup products of elementary preorders whose transitive-closure product is basic---as in the first \red{assertion} in \red{Remark}~\ref{pdd2}.
\end{theorem}

Cup products of elementary preorders whose corresponding transitive-closure product fails to be basic can be written in terms of basic ones by iterated use of the first relation in Theorem~\ref{cohbar}. The process is clarified in the next section, where we work extensively in terms of Baryshnikov's basis in $H^*(\confkRn)$, and the corresponding tensor basis in $H^*(\confkRn\times \confkRn) \cong H^*(\confkRn)\otimes H^*(\confkRn)$.

\section{$\TC(\confkRn)$}\label{secciondetcf}
Theorem~\ref{teoremaprincipal} is obvious for $n\le k$. In fact, for $n<k$, $\confkRn=\mathbb{R}^n$, which is contractible, so that $\TC(\confkRn)=0$. On the other hand, for $n=k$ and with $\Delta=\{(x,x,\ldots,x)\colon x\in\mathbb{R}\}$, $\F_k(\mathbb{R},n)=\mathbb{R}^k-\Delta\simeq S^{k-2}$, whose topological complexity is well known to be $1$ (respectively $2$) if $k$ is odd (respectively even). We thus assume {$n>k$} in what follows (recall we also assume $k\ge3$).

The homotopy dimension (hdim) and the Luster\-nik-Schnirelmann category (cat) of $\confkRn$ (and thus the assertion in Theorem~\ref{teoremaprincipal} about the latter number) are easily established:

\begin{lemma}\label{dimension}
$\confkRn$ is a ($k-3$)-connected space having $\cat(\confkRn)=\lfloor n/k\rfloor$ and $\hdim(\confkRn)=(k-2)\lfloor n/k\rfloor$. In particular, for $k=3$, both the cohomological dimension ($\cdim$) and the geometric dimension ($\gdim$) of the group $\ppn$ equal $\lfloor n/k\rfloor$.
\end{lemma}
\begin{proof}
Let $q=\lfloor n/k\rfloor$. The Baryshnikov basis element
\begin{equation}\label{bbe}
\left[1,\dots, k-1\rule{0mm}{4mm}\right]\left(k\rule{0mm}{4mm}\right)\left[k+1,\dots, 2k-1\rule{0mm}{4mm}\right]\left(2k\rule{0mm}{4mm}\right)\cdots\left[(q-1)k+1,\ldots,qk-1\rule{0mm}{4mm}\right]\left(qk\rule{0mm}{4mm}\right)
\end{equation}
is a (non-zero) product of $q$ factors, each being a dimension-1 basis element (see the first \red{assertion} in \red{Remark}~\ref{pdd2}), which implies $q\le\cat(\confkRn)$. On the other hand,~\cite[Theorems~1.1 and~1.2]{MR2915650} imply that $\confkRn$ is $(k-3)$-connected, is not $(k-2)$-connected, and has the homotopy type of a cell complex of dimension $(k-2)q$. The first two assertions in the lemma then follow from the inequality $\cat\le(\hdim)/(\conn+1)$ ---which in turn follows from a standard obstruction-theory argument. The last assertion in the lemma (for $k=3$, so $\hdim(\confkRn)=\gdim(\ppn)$, by definition), follows from the relations $\cat=\cdim\le\gdim$ in~\cite{MR0085510}.
\end{proof}

We have omitted the use of curly braces for level sets within the string preorder~(\ref{bbe}). This convention will be kept throughout the rest of the paper.

The standard inequality $\TC(X)\leq2\cat(X)$ yields $\TC(\confkRn)\leq 2\lfloor n/k\rfloor$. Thus, in view of Proposition~\ref{ulbTCn}, the proof of Theorem~\ref{teoremaprincipal} will be complete once we show
\begin{equation}\label{loquefalta}
2\lfloor n/k \rfloor\leq\zcl(\confkRn), \mbox{ \ for $n>k\geq3$.}
\end{equation}

In order to \red{establish}~(\ref{loquefalta}), we introduce a few key elements in $H^*(\confkRn)$ and in $H^*(\confkRn)^{\otimes2}$. (Recall that all cohomology groups will be taken with $\mathbb{Z}_2$-coefficients, a restriction that is not essential but allows us to simplify calculations.)

\begin{definition}\label{xmxpm}
For a positive integer $m$ satisfying $m+k\le n+2$, consider the elements $x_m,x'_m\in H^{k-2}(\confkRn)$ given by
\begin{align*}
x_m{}={}&\left(1,\ldots,m-2,m-1\rule{0mm}{4mm}\right)\left[\rule{0mm}{4mm}m,m+1,\dots, m+k-2\right]\left(\rule{0mm}{4mm}m+k-1,\ldots,n\right), \\
x'_m{}={}&\left(\rule{0mm}{4mm}1,\ldots,m-2,m\right)\left[\rule{0mm}{4mm}m-1,m+1,\dots, m+k-2\right]\left(\rule{0mm}{4mm}m+k-1,\ldots,n\right),
\end{align*}
where $x'_m$ is defined only for $m\geq2$. Each of the corresponding zero-divisor $y_m=x_m\otimes1+1\otimes x_m$ for $\confkRn$ is central in what follows, with the elements $x'_m$ playing a subtle role.
\end{definition}

Note that $x_m$ and $x'_m$ are Baryshnikov basis elements in $H^*(\confkRn)$ provided $m+k\le n+1$. In fact, as illustrated by the first \red{assertion} in \red{Remark}~\ref{pdd2},
\begin{align}
\hspace{3mm} &\hspace{-3mm}\prod_{j=1}^i x_{(j-1)k+2}=x_2x_{k+2}\cdots x_{(i-1)k+2}\label{mas2}\\&=\left(\rule{0mm}{3.5mm}1\right)\!\left[\rule{0mm}{3.5mm}2,\ldots,k\right]\!\left(\rule{0mm}{3.5mm}k+1\right)\!\left[\rule{0mm}{3.5mm}k+2,\ldots,2k\right]\!\left(2k+1\rule{0mm}{3.5mm}\right)\cdots\left[\rule{0mm}{3.5mm}(i-1)k+2,\ldots,ik\right]\!\left(\rule{0mm}{3.5mm}ik+1,\ldots,n\right)\nonumber
\end{align}
is a basis element in $H^*(\confkRn)$ provided $ik+1\leq n$. Likewise, if $\widetilde{x}_{(j-1)k+1}$ stands for either $x_{(j-1)k+1}$ or $x'_{(j-1)k+1}$ (the latter one being a possibility only for $j\geq2$), then
\begin{align}
\hspace{3mm} &\hspace{-3mm}\prod_{j=1}^i \widetilde{x}_{(j-1)k+1}=\widetilde{x}_1\widetilde{x}_{k+1}\cdots\widetilde{x}_{(i-1)k+1}\label{mas1}\\&=\left[\rule{0mm}{3.5mm}1,\!\cdots\hspace{-.2mm}\hspace{-.5mm},k-1\right]\hspace{-.7mm}\left(\rule{0mm}{3.5mm}k\begin{picture}(0,0)(4,8)\qbezier(3,0)(13,-20)(26,0)\put(3,0){\vector(-1,2){2}}\put(26,0){\vector(1,2){2}}\end{picture}\right)\hspace{-.7mm}\left[\rule{0mm}{3.5mm}k+1,\!\cdots\hspace{-.2mm}\hspace{-.5mm},2k-1\right]\!\cdots\hspace{-.3mm}\left(\rule{0mm}{3.5mm}(i-1)k\begin{picture}(0,0)(4,8)\qbezier(-9,0)(13,-20)(39,0)\put(-9,0){\vector(-1,1){2}}\put(39,0){\vector(1,1){2}}\end{picture}\right)\hspace{-.7mm}\left[\rule{0mm}{3.5mm}(i-1)k+1,\cdots\hspace{-.5mm},ik-1\right]\hspace{-.7mm}\left(\rule{0mm}{3.5mm}ik,\ldots,n\right),\nonumber
\end{align}

\vspace{4mm}\noindent
where curved arrows indicate pairs of elements that might have to be switched (depending on the actual term $\widetilde{x}_{(j-1)k+1}$ under consideration), is a basis element in $H^*(\confkRn)$ provided $ik\leq n$.

\begin{example}\label{caschi}{\em
The condition $3\le k<n$ ensures that both $x_1$ and $x_2$ are Baryshnikov basis elements in $H^*(\confkRn)$, and since $x_1\neq x_2$, we obviously have
\begin{equation}\label{facilito}
y_1y_2=(x_1\otimes1+1\otimes x_1)(x_2\otimes1+1\otimes x_2)=\cdots+x_1\otimes x_2+x_2\otimes x_1+\cdots\neq0.
\end{equation}
So $2\leq\zcl(\confkRn)$, which readily yields~(\ref{loquefalta}) for $2k>n>k\ge3$.
}\end{example}

The proof of~(\ref{loquefalta}) for $n\ge2k$ and $k\ge3$ requires a major generalization of the simple calculation in~(\ref{facilito}). The product indicated in~(\ref{productotote}) below will play the role of the product $y_1y_2$ on the left-hand side of~(\ref{facilito}). Most importantly, the tensor factors $x_1$ and $x_2$ in the two highlighted summands on the right-hand side of~(\ref{facilito}) will be replaced by products of the form~(\ref{mas2}), and by certain products of the form~(\ref{mas1}), some of which are made explicit as follows:
\begin{align*}
p_{i,1}{}={}&\begin{cases}
x_1\left(\prod_{j=1}^{a-1} x_{(2j-1)k+1}x'_{2jk+1}\right) x_{(2a-1)k+1}\,, & \mbox{ if $i=2a\geq2$;} \\
x_1\left(\prod_{j=1}^{a} x_{(2j-1)k+1}x'_{2jk+1}\right), & \mbox{ if $i=2a+1\geq3$,}
\end{cases}\\
p_{i,2}{}={}&\begin{cases}
x_1\left(\prod_{j=1}^{a-1} x'_{(2j-1)k+1}x_{2jk+1}\right) x'_{(2a-1)k+1}\,, & \mbox{ if $i=2a\geq2$;} \\
x_1\left(\prod_{j=1}^{a} x'_{(2j-1)k+1}x_{2jk+1}\right), & \mbox{ if $i=2a+1\geq3$.}
\end{cases}\end{align*}

\begin{theorem}\label{detallado} If the integers $i,k,n$ satisfy $2\le i$, $3\le k$, and $ik\leq n$, then the product
\begin{equation}\label{productotote}
\prod_{j=1}^i y_{(j-1)k+1} y_{(j-1)k+2}\in H^*(\confkRn)^{\otimes2}
\end{equation}
is non-zero. Explicitly:
\begin{enumerate}
\item If $ik+1\le n$, then the expression of~(\ref{productotote}) as a linear combination of Baryshnikov tensor basis elements for $H^*(\confkRn)^{\otimes2}$ uses the Baryshnikov basis element $$\,\prod_{j=1}^i x_{(j-1)k+1}\otimes\prod_{j=1}^ix_{(j-1)k+2}.$$
\item If $ki=n$, then the expression of~(\ref{productotote}) as a linear combination of Baryshnikov tensor basis elements for $H^*(\confkRn)^{\otimes2}$ uses the Baryshnikov basis element $p_{i,1}\otimes p_{i,2}$.
\end{enumerate}
\end{theorem}

As distilled in Example~\ref{caschi}, the hypothesis $i\geq2$ is relevant only for the second half of Theorem~\ref{detallado}. The actual exceptional case that has to be avoided is $n=k$, for which $y_1y_2$ is forced to vanish (recall the $\mathbb{Z}_2$-coefficients!) in view of the first paragraph of this section. (By working over the integers, rather than over $\mathbb{Z}_2$, the (truly!) exceptional case would only be reduced to that where $n=k$ is odd.)

The validness of~(\ref{loquefalta}) for $n\ge2k$ and $k\ge3$ (i.e.~the cases that remain to be considered) follows from Theorem~\ref{detallado} below by taking $i=\lfloor n/k \rfloor$. So, the rest of the section is devoted to the proof of Theorem~\ref{detallado}.

\begin{lemma}\label{auxilema}
The following relations hold in $H^*(\confkRn)$$:$
\begin{enumerate}
\item\label{rel0} $x_2x_{k+1}=x_1x_{k+1}$, for $n\geq2k-1$.
\item\label{rel1} $x_{n-2k+4} x_{n-k+2}=x_{n-2k+3} x_{n-k+2}=0$, for $n\ge2k-2$.
\item\label{rel2} $x_{n-2k+2} x_{n-k+2} = x_{n-2k+2} x_{n-k+1}$, for $n\ge2k-1$.
\item\label{rel3} $x_{n-2k+1} x_{n-k+2} = x_{n-2k+1} x_{n-k+1} + x_{n-2k+1} x'_{n-k+1}$, for $n\ge2k$.
\item\label{rel4} $x_r x_{r+k} x_{r+2k-1} = x_r x_{r+k-1} x_{r+2k-1}$, for $n\ge r+3k-3$ and $r\ge1$.
\item\label{rel5} $x_r x_{r+k+1} x_{r+2k} = x_r x_{r+k} x_{r+2k} + x_r x'_{r+k} x_{r+2k}$, for $n\ge r+3k-2$ and $r\ge1$.
\end{enumerate}
\end{lemma}
\begin{remark}{\em
The numeric restrictions on $k$, $n$ and $r$ ensure that each of the factors $x_m$ in the six items above is an element of $H^*(\confkRn)$.
}\end{remark}
\begin{proof}[Proof of Lemma~\ref{auxilema}]
All these equalities follow from Theorem~\ref{cohbar} and Remark~\ref{cuandocero}. We give full details for completeness.

Assume $n\ge2k-2$. Take $A=\{1,\ldots,n-k+1\}$, $B=\{n-k+2,\dots, n-1\}$ and $C=\{n\}$ in Theorem~\ref{cohbar}.1 to get
\begin{align}\label{reuso}
x_{n-k+2}&=(1,\ldots,n-k+1)[n-k+2,\dots, n]\nonumber\\
&=\sum_{i=1}^{n-k+1}(1,\ldots,\widehat{i},\ldots,n-k+1)[i,n-k+2,\dots, n-1](n).
\end{align}
As explained in Remark~\ref{cuandocero}, all terms in the summation in~(\ref{reuso}) vanish when multiplied by $x_{n-2k+3}=(1,\ldots,n-2k+2)[n-2k+3,\dots, n-k+1](n-k+2,\cdots,n)$. This yields $x_{n-2k+3} x_{n-k+2}=0$, while the equality $x_{n-2k+4}x_{n-k+2}=0$ follows directly from the considerations in Remark~\ref{cuandocero}. This proves item~\ref{rel1}.

Assume $n\geq2k-1$. Terms with $i\le n-k$ in the summation in~(\ref{reuso}) vanish when multiplied by $x_{n-2k+2}=(1,\ldots,n-2k+1)[n-2k+2,\dots, n-k](n-k+1,\dots,n)$. This yields $x_{n-2k+2}x_{n-k+2}=x_{n-2k+2}x_{n-k+1}$, proving item~\ref{rel2}.

Assume $n\geq2k$. Terms with $i<n-k$ in the summation in~(\ref{reuso}) vanish when multiplied by $x_{n-2k+1}=(1,\ldots,n-2k)[n-2k+1,\dots, n-k-1](n-k,\dots,n)$. This yields $x_{n-2k+1}x_{n-k+2}=x_{n-2k+1}x_{n-k+1}+x_{n-2k+1}x'_{n-k+1}$, proving item~\ref{rel3}.

Assume $n\ge r+3k-2$ and $r\ge1$. Take $A=\{1,\ldots,r+k-1\}$, $B=\{r+k,\dots, r+2k-3\}$ and $C=\{r+2k-2,\ldots,n\}$ in Theorem~\ref{cohbar}.1 to get
\begin{multline*}
\sum_{i=1}^{r+k-1}(1,\ldots,\widehat{i},\ldots,r+k-1)[i,r+k,\dots, r+2k-3](r+2k-2,\ldots,n)\\=\sum_{i=r+2k-2}^{n}(1,\ldots,r+k-1)[r+k,\dots, r+2k-3,i](r+2k-2,\ldots,\widehat{i},\ldots,n).
\end{multline*}
Terms with $i<r+k-1$ in the first summation vanish when multiplied by $x_{r}=(1,\ldots,r-1)[r,\ldots,r+k-2](r+k-1,\dots, n)$, and terms with $i>r+2k-2$ in the second summation vanish when multiplied by $x_{r+2k-1}=(1,\ldots,r+2k-2)[r+2k-1,\ldots,r+3k-3](r+3k-2,\dots, n)$. This yields the equality $x_r x_{r+k-1} x_{r+2k-1} = x_r x_{r+k} x_{r+2k-1}$, proving item~\ref{rel4}.

When $n\ge2k-1$, the previous argument applies for $r=2-k$ ---by vacuity in the case of the assertion about the first summation, whose only one term is~$x_1$. This yields $x_1x_{k+1}=x_2x_{k+1}$, proving item~\ref{rel0}.

Assume $n\ge r+3k-2$ and $r\ge1$. Take $A=\{1,\ldots,r+k\}$, $B=\{r+k+1,\dots, r+2k-2\}$ and $C=\{r+2k-1,\ldots,n\}$ in Theorem~\ref{cohbar}.1 to get
\begin{multline}
\sum_{i=1}^{r+k}(1,\ldots,\widehat{i},\ldots,r+k)[i,r+k+1,\dots,r+2k-2](r+2k-1,\ldots,n)\\=\sum_{i=r+2k-1}^{n}(1,\ldots,r+k)[r+k+1,\dots, r+2k-2,i](r+2k-1,\ldots,\widehat{i},\ldots,n).
\end{multline}
Terms with $i<r+k-1$ in the first summation vanish when multiplied by $x_r=(1,\ldots,r-1)[r,\dots, r+k-2](r+k-1,\ldots,n)$, while terms with $i>r+2k-1$ in the second summation vanish when multiplied by $x_{r+2k}=(1,\ldots,r+2k-1)[r+2k,\dots,r+3k-2](r+3k-1,\ldots,n)$. This yields the equality $x_r x'_{r+k} x_{r+2k}+x_r x_{r+k} x_{r+2k}=x_r x_{r+k+1} x_{r+2k}$, proving item~\ref{rel5}.
\end{proof}

\begin{proof}[Proof of part 1 in Theorem~\ref{detallado}]
By Remark~\ref{cuandocero},
\begin{align*}
y_{(j-1)k+1}y_{(j-1)k+2}={}&(x_{(j-1)k+1}\otimes1+1\otimes x_{(j-1)k+1})(x_{(j-1)k+2}\otimes1+1\otimes x_{(j-1)k+2})\\
{}={}&x_{(j-1)k+1}\otimes x_{(j-1)k+2}+x_{(j-1)k+2}\otimes x_{(j-1)k+1},
\end{align*}
so the product in~(\ref{productotote}) is
\begin{align}
\prod_{j=1}^iy_{(j-1)k+1}y_{(j-1)k+2}&{}=(x_1\otimes x_2+x_2\otimes x_1)(x_{k+1}\otimes x_{k+2}+x_{k+2}\otimes x_{k+1})\ \cdots\nonumber\\[-3pt]
&\hspace{2.5cm}\cdots\ (x_{(i-1)k+1}\otimes x_{(i-1)k+2}+x_{(i-1)k+2}\otimes x_{(i-1)k+1})\nonumber\\[4pt]
&{}=\sum_{\mbox{\scriptsize$\begin{array}{c}\epsilon_j\in\{1,2\} \\ 1\leq j\leq i\end{array}$}}\!\!\!\!\!x_{3-\epsilon_1}x_{k+3-\epsilon_2}\cdots x_{(i-1)k+3-\epsilon_i}\otimes x_{\epsilon_1}x_{k+\epsilon_2}\cdots x_{(i-1)k+\epsilon_i}.\label{adesarrollar}
\end{align}
The basis element we care about, namely
\begin{equation}\label{rastrear}
\prod_{j=1}^ix_{(j-1)k+1}\otimes\prod_{j=1}^ix_{(j-1)k+2},
\end{equation}
is the summand in~(\ref{adesarrollar}) with $\epsilon_j=2$ for all $j$.
The proof task is to argue that, when we expand the other terms of~(\ref{adesarrollar}) as sums of tensor of basis elements, the tensor~(\ref{rastrear}) does not appear.
This is obvious for the summand in~(\ref{adesarrollar}) with $\epsilon_j=1$ for all $j$. For all other summands, the assertion will be argued by focusing on the sequence of leaps associated to the subscripts of both tensor factors of each summand in~(\ref{adesarrollar}). Explicitly, the first leap in the subscripts of $x_{3-\epsilon_1}x_{k+3-\epsilon_2}\cdots x_{(i-1)k+3-\epsilon_i}$ is\hspace{.5mm} $k+3-\epsilon_2-(3-\epsilon_1)=k+\epsilon_1-\epsilon_2$, and the full sequences of leaps associated to
\begin{equation}\label{factores}
x_{3-\epsilon_1}x_{k+3-\epsilon_2}\cdots x_{k(i-1)+3-\epsilon_i}\quad\mbox{and}\quad x_{\epsilon_1}x_{k+\epsilon_2}\cdots x_{k(i-1)+\epsilon_i}
\end{equation}
are, respectively, 
\begin{equation}\label{seqlea}
(k+\epsilon_1-\epsilon_2, k+\epsilon_2-\epsilon_3,\ldots,k+\epsilon_{i-1}-\epsilon_i)\quad\mbox{and}\quad (k-\epsilon_1+\epsilon_2, k-\epsilon_2+\epsilon_3,\ldots,k-\epsilon_{i-1}+\epsilon_i).
\end{equation}
Such sequences of leaps clearly satisfy:
\begin{itemize}
\item[(A)] Leap values are either $k-1$, $k$, or $k+1$. Moreover, if all $k$-leaps are removed from either one of the sequences in~(\ref{seqlea}), then the resulting sequence of leaps either is empty or, else, has leap values that alternate between $k-1$ and $k+1$: ($k-1$,$k+1$,$k-1$,\ldots) or ($k+1$,$k-1$,$k+1$,\ldots).
\item[(B)] The two sequences of leaps in~(\ref{seqlea}) are coordinate-wise complementary to each other with respect to $2k$.
\item[(C)] The first leap different from $k$ (if any) in either of the sequences of leaps~(\ref{seqlea}) is a ($k+1$)-leap (($k-1$)-leap) provided the corresponding product in~(\ref{factores}) starts with $x_1$ ($x_2$).
\end{itemize}
Since the right tensor factor in~(\ref{rastrear}), i.e.~$\prod_{j=1}^ix_{(j-1)k+2}$, is a basic string preorder starting as $(1)[2,\dots,k]\cdots$, the proof is complete in view of Proposition~\ref{traduccion} below.
\end{proof}

\begin{proposition}\label{traduccion}
Any summand in~(\ref{adesarrollar}) whose associated sequences of leaps~(\ref{seqlea}) contain at least a $(k-1)$-leap (equivalently a $(k+1)$-leap) is a linear combination of tensor basis elements $u\otimes v$ where \emph{both} $u$ and $v$ are basic string preorders starting as $$[1,\dots,k-1](I_1)\cdots(I_{i-1})[J_i](I_i).$$
\end{proposition}
\begin{proof}
Take a product $p=x_{k_1}x_{k_2}\cdots x_{k_i}$ in~(\ref{factores}), so $k_1\in\{1,2\}$, with associated sequence of leaps $(\ell_1,\ldots,\ell_{i-1})$ satisfying conditions (A)--(C) above, and so that not all leap values $\ell_j$ are $k$.

\noindent {\bf Case $k_1=1$:} $p$ has the form
\begin{equation}\label{forma1}
x_1\cdots\underbrace{x_{kr_1+1}x_{k(r_1+1)+2}}_{\text{($k+1$)-leap}}\cdots\underbrace{x_{kr_2+2}x_{k(r_2+1)+1}}_{\text{($k-1$)-leap}}\cdots\underbrace{x_{kr_3+1}x_{k(r_3+1)+2}}_{\text{($k+1$)-leap}}\cdots\underbrace{x_{kr_4+2}x_{k(r_4+1)+1}}_{\text{($k-1$)-leap}}\cdots,
\end{equation}
where we only indicate ($k-1$)-leaps and ($k+1$)-leaps. Items~\ref{rel4} and~\ref{rel5} in Lemma~\ref{auxilema} allow us to replace each portion $x_{kr_j+1}x_{k(r_j+1)+2}\cdots x_{kr_{j+1}+2}x_{k(r_{j+1}+1)+1}$, having an initial ($k+1$)-leap, a final ($k-1$)-leap, and (perhaps) some intermediate $k$-leaps, by $$x_{kr_j+1}(x_{k(r_j+1)+1}+x'_{k(r_j+1)+1})x_{k(r_j+2)+1}\cdots x_{kr_{j+1}+1}x_{k(r_{j+1}+1)+1},$$
which only has $k$-leaps. The replacing process can be iterated since the initial and final terms in the replacing portion agree with those in the replaced portion. After all replacements are made, and sums are distributed, $p$ becomes a sum of expressions each of which is similar to the original one~(\ref{forma1}), except that some of the initial $x_{kj+1}$'s get replaced by the corresponding $x'_{kj+1}$, and in such a way that no ($k-1$)-leaps show up, and at most one ($k+1$)-leap shows up. But any such expression is a basis element of the required form (the latter assertion uses the hypothesis $ik+1\leq n$ in part 1 of Theorem~\ref{detallado} ---see Remark~\ref{aclarando} below).

\noindent {\bf Case $k_1=2$:} $p$ has the form
\begin{equation*}
x_2\cdots\underbrace{x_{kr_1+2}x_{k(r_1+1)+1}}_{\text{($k-1$)-leap}}\cdots\underbrace{x_{kr_2+1}x_{k(r_2+1)+2}}_{\text{($k+1$)-leap}}\cdots\underbrace{x_{kr_3+2}x_{k(r_3+1)+1}}_{\text{($k-1$)-leap}}\cdots\underbrace{x_{kr_4+1}x_{k(r_4+1)+2}}_{\text{($k+1$)-leap}}\cdots,
\end{equation*}
Items~\ref{rel0} and~\ref{rel4} in Lemma~\ref{auxilema} allow us to replace the initial portion $x_2\cdots x_{kr_1+2}x_{k(r_1+1)+1}$ by $x_1\cdots x_{kr_1+1}x_{k(r_1+1)+1}$. Then, the replacement process described in the previous case allows us to write $p$ as a sum of basis elements of the required form.
\end{proof}

\begin{remark}\label{aclarando}{\em
Part 2 in Theorem~\ref{detallado} will be proved using an argument similar to that in the previous proof, except that it will be necessary to deal first with an additional subtlety. Namely, note that when $ik=n$, we have
\begin{align*}
x_{(i-1)k+2}={}&\left(\rule{0mm}{3.5mm}1,\cdots,(i-1)k+1\right)\left[\rule{0mm}{3.5mm}(i-1)k+2,\cdots,ik\right]\left(\rule{0mm}{3.5mm}ik+1,\cdots,n\right)\\
{}={}&\left(\rule{0mm}{3.5mm}1,\cdots,(i-1)k+1\right)\left[\rule{0mm}{3.5mm}(i-1)k+2,\cdots,n\right],
\end{align*}
which is an elementary \emph{non-basic} element (i.e., under the main hypothesis in part~2 of Theorem~\ref{detallado}). So, when analyzing a typical tensor factor $x_{\epsilon_1}x_{k+\epsilon_2}\cdots x_{(i-1)k+\epsilon_i}$ in~(\ref{adesarrollar}) with $\epsilon_i=2$, the recursive process described in the previous proof will not end up producing sums of basis elements. This issue will be resolved using item~\ref{rel3} in Lemma~\ref{auxilema}.}\end{remark}

Let us go back to the starting point for the proof of part 2 in Theorem~\ref{detallado}, i.e., the expression in~(\ref{adesarrollar}) for the product $\prod_{j=1}^iy_{(j-1)k+1}y_{(j-1)k+2}$. As observed in Remark~\ref{aclarando}, we no longer work \red{with} the basis element indicated in part 1 of Theorem~\ref{detallado}. Instead, the basis element we now care about is $p_{i,1}\otimes p_{i,2}$, where $ki=n$, and which arises from one of the two summands in~(\ref{adesarrollar}) for which the values of the indices $\epsilon_j$ alternate between 1 and 2. 

In order to simplify the argument, it is convenient to note that all $y_j$, and therefore their product $\prod_{j=1}^iy_{(j-1)k+1}y_{(j-1)k+2}$, are invariant under the involution induced by the map that switches coordinates in $\confkRn\times\confkRn$. We show the following (equivalent, by the symmetry just noted, but slightly simpler-to-prove) version of part 2 in Theorem~\ref{detallado}:

\begin{theorem}
For $i\ge2$, $k\ge3$ and $n=ki$, both $p_{i,1}\otimes p_{i,2}$ and $p_{i,2}\otimes p_{i,1}$ are used in the expression of the product~(\ref{productotote}) as a linear combination of Baryshnikov tensor basis elements for $H^*(\confkRn)^{\otimes2}$.
\end{theorem}
\begin{proof}
We provide full proof details when $i=2a$ is even; the parallel argument for $i$ odd is left as an exercise for the reader. In order to simplify notation, we let $r_1\pt r_2\cdots r_t$ and $r_1\pt r_2\cdots r_t | s_1\pt s_2\cdots s_t$ stand for $x_{r_1}x_{r_2}\cdots x_{r_t}$ and $x_{r_1}x_{r_2}\cdots x_{r_t}\otimes x_{s_1}x_{s_2}\cdots x_{s_t}$, respectively. With this notation,~(\ref{adesarrollar}) becomes
\small\begin{align}
&{}\left(1|2+2|1\rule{0mm}{4mm}\right)\left(\rule{0mm}{4mm}(k+1)|(k+2)+(k+2)|(k+1)\right)\ \cdots \nonumber\\
& \hspace{3cm}\cdots\ \left(((2a-1)k+1)|((2a-1)k+2)+((2a-1)k+2)|((2a-1)k+1)\rule{0mm}{4mm}\right)\nonumber\\
&{}=\hspace{-0.3cm}\sum_{\mbox{\scriptsize$\begin{array}{c}\epsilon_j\in\{1,2\} \\ 1\leq j\leq i\end{array}$}}\!\!\!\!\!(3-\epsilon_1)(k+3-\epsilon_2)\cdots((2a-1)k+3-\epsilon_{2a})\left|\rule{0mm}{4mm}\right.(\epsilon_1)(k+\epsilon_2)\cdots((2a-1)k+\epsilon_{2a}) .\label{adesarrollarbis}
\end{align}
The summand with $(\epsilon_1,\epsilon_2,\cdots,\epsilon_{2a})=(1,2,1,\ldots,2)$ is
\begin{multline}\label{f1}
2\pt (k+1)\pt (2k+2)\pt (3k+1)\cdots((2a-2)k+2)\pt((2a-1)k+1)\\\left|\rule{0mm}{4mm}\right.1\pt (k+2)\pt (2k+1)\pt (3k+2)\cdots((2a-2)k+1)\pt((2a-1)k+2),
\end{multline}
whose associated sequences of leaps are
\begin{equation}\label{associatedleaps}
(k-1,k+1,k-1,\ldots,k-1)\quad\mbox{and}\quad(k+1,k-1,k+1,\ldots,k+1).
\end{equation}
Using the replacing process explained in the previous proof, it is clear that the expression of $$2\pt (k+1)\pt (2k+2)\pt (3k+1)\cdots((2a-2)k+2)\pt((2a-1)k+1)$$ in terms of Baryshnikov basis elements uses $p_{2a,1}$, but not $p_{2a,2}$. Likewise, the replacing process and item~\ref{rel3} in Lemma~\ref{auxilema} imply that the expression of $$1\pt (k+2)\pt (2k+1)\pt (3k+2)\cdots((2a-2)k+1)\pt((2a-1)k+2)$$ in terms of Baryshnikov basis uses $p_{2a,2}$. Therefore the expression of~(\ref{f1}) in terms of Baryshnikov (tensor) basis elements uses $p_{2a,1}\otimes p_{2a,2}$ without using $p_{2a,2}\otimes p_{2a,1}$. Further, the symmetry coming from the involution induced by the switching map on $\confkRn^{\times2}$ implies that the expression in terms of Baryshnikov basis of the summand in~(\ref{adesarrollarbis}) with $(\epsilon_1,\epsilon_2,\cdots,\epsilon_{2a})=(2,1,2\ldots,1)$ uses $p_{2a,2}\otimes p_{2a,1}$ without using $p_{2a,1}\otimes p_{2a,2}$.

It remains to prove that neither $p_{2a,1}\otimes p_{2a,2}$ nor $p_{2a,2}\otimes p_{2a,1}$ are used in the expression in terms of basis elements of any summand in~(\ref{adesarrollarbis}) whose associated sequences of leaps is different from those in~(\ref{associatedleaps}). By symmetry, it suffices to consider the case of a summand 
\begin{equation}\label{suffices}
(3-\epsilon_1)(k+3-\epsilon_2)\cdots((2a-1)k+3-\epsilon_{2a})\left|\rule{0mm}{4mm}\right.(\epsilon_1)(k+\epsilon_2)\cdots((2a-1)k+\epsilon_{2a})
\end{equation}
with $\epsilon_1=1$. Let $\lambda\in\{k-1,k,k+1\}$ ($\rho\in\{k+1,k,k-1\}$) stand for the value of the last leap in the tensor factor on the left (right) of~(\ref{suffices}). Recall $\lambda+\rho=2k$.

\noindent {\bf Case $\lambda=\rho=k$:} The ending portion of one of the two tensor factors in~(\ref{suffices}) is forced to be $$\cdots\!((2a-2)k+1)\pt((2a-1)k+1).$$ The replacing process shows that such a factor cannot give rise to $p_{2a,1}$ or $p_{2a,2}$ in its expression in terms of Baryshnikov basis.

\noindent {\bf Case $(\lambda,\rho)=(k-1,k+1)$:} The equalities $\epsilon_{2a-1}=1$ and $\epsilon_{2a}=2$ are now forced. Letting $j'$ stand for $x'_j$, and ignoring Baryshnikov basis elements different from $p_{2a,1}$ and $p_{2a,2}$, the right factor in~(\ref{suffices}) then becomes
$$
1\pt(k+\epsilon_2)\cdots((2a-2)k+1)((2a-1)k+2)=1\pt(k+\epsilon_2)\cdots((2a-2)k+1)((2a-1)k+1)',
$$
in view of the replacing process and item~\ref{rel3} in Lemma~\ref{auxilema}. Further, the replacing process makes it clear that the expression of the latter element in terms of Baryshnikov basis elements does not use $p_{2a,1}$, and that it uses $p_{2a,2}$ only if the sequence of leaps associated to the right tensor factor in~(\ref{suffices}) is the second sequence in~(\ref{associatedleaps}). 

\noindent {\bf Case $(\lambda,\rho)=(k+1,k-1)$:} The equalities $\epsilon_{2a-1}=2$ and $\epsilon_{2a}=1$ are now forced. Ignoring Baryshnikov basis elements different from $p_{2a,1}$ and $p_{2a,2}$, the left factor in~(\ref{suffices}) becomes
\begin{align*}
2\pt(k+3-\epsilon_2)\cdots((2a-2)k+1)&((2a-1)k+2) \\&=2\pt(k+3-\epsilon_2)\cdots((2a-2)k+1)((2a-1)k+1)',
\end{align*}
where the latter expression further evolves under the replacing process (still ignoring Baryshnikov basis elements different from $p_{2a,1}$ and $p_{2a,2}$) to either zero or to 
\begin{equation}\label{yacasisitito}
2\pt (k+1) \pt (2k+1) \pt (3k+1)'\cdots((2a-2)k+1)((2a-1)k+1)'.
\end{equation}
Note the factor ``$(k+1)$'', rather than a (primed) ``$(k+1)'\hspace{.2mm}$'', due to the initial ``$2$'' in~(\ref{yacasisitito}). In any case, a final application of item~\ref{rel0} in Lemma~\ref{auxilema} shows that~(\ref{yacasisitito}) vanishes modulo Baryshnikov basis elements different from $p_{2a,1}$ and $p_{2a,2}$.
\end{proof}

\section{The higher topological complexity of $\confkRn$}
We can now easily deduce the value of the higher topological complexity $\TC_s(\confkRn)$, for any $s>2$.
\begin{corollary}\label{lassecuenciales} For $s>2$, 
\[
\TC_s(\confkRn) = \begin{cases}
					0, & n < k; \\
					s-1,& n=k \mbox{ with $k$ odd;} \\
					s, & n=k  \mbox{ with $k$ even;}\\
					s\lfloor n/k\rfloor, & n>k.
				\end{cases}
\]
\begin{proof}
The case $n\leq k$ is trivial. For $n>k$ and $s>2$, Lemma~\ref{dimension} \red{and~\cite[Theorem~3.9]{bgrt}} imply the estimate $\TC_{\red{s}}(\confkRn) \leq s\lfloor n/k\rfloor$. From~\cite[Definition~3.8 and Theorem~3.9]{bgrt}, equality will follow once we exhibit a non-zero product of $s\lfloor n/k\rfloor$ ``$s$-th zero-divisors'' for $\confkRn$, i.e., of elements in the kernel of the iterated cup product $H^*(\confkRn)^{\otimes s}\to H^*(\confkRn)$.

Let $i=\left\lfloor n/k\right\rfloor$, $q\in\{1,\ldots,s-1\}$, and consider the $s$-th zero-divisors
$$
z_{m,q} = 1\otimes \cdots \otimes 1\otimes \underbrace{x_m}_{q\text{-th}} \otimes 1 \otimes  \cdots \otimes 1 + 1\otimes \cdots \otimes 1 \otimes x_m\in H^*(\confkRn)^{\otimes s},
$$
whenever $m+k\le n+2$. For instance
$$\prod_{j=1}^i z_{(j-1)k+1,s-1}z_{(j-1)k+2,s-1} = 1\otimes \cdots \otimes 1\otimes \prod_{j=1}^i y_{(j-1)k+1} \cdot y_{(j-1)k+2}$$ and, for $q\le s-2$,
\begin{align*}
z_{m, q}\prod_{j=1}^i z_{(j-1)k+1,s-1}&z_{(j-1)k+2,s-1}\\ &{}= 1\otimes \cdots \otimes 1 \otimes \underbrace{x_m}_{q-\text{th}} \otimes 1 \otimes \cdots \otimes 1 \otimes \prod_{j=1}^i y_{(j-1)k+1} \cdot y_{(j-1)k+2}\\&\ {}+1\otimes \cdots\otimes1\otimes\left((1\otimes x_m)\cdot \prod_{j=1}^i y_{(j-1)k+1}y_{(j-1)k+2}\right).
\end{align*}
The second summand in the latter expression vanishes in view of Lemma~\ref{dimension} (by dimensional considerations or, alternatively, by $\cat$-considerations). Consequently
\begin{align*}
\prod_{j=1}^{i} z_{(j-1)k+1,1} \:\cdot & \prod_{j=1}^{i} z_{(j-1)k+1,2}\: \cdots \prod_{j=1}^{i}  z_{(j-1)k+1,s-2} \:\cdot \prod_{j=1}^{i} z_{(j-1)k+1,s-1} z_{(j-1)k+2,s-1} \\
&= \left(\,\prod_{j=1}^i x_{(j-1)k+1}\right) \otimes \cdots \otimes \left(\,\prod_{j=1}^i x_{(j-1)k+1}\right) \otimes \prod_{j=1}^i y_{(j-1)k+1}y_{(j-1)k+2},
\end{align*}
which is non-zero because the first $s-2$ tensor factors in the latter expression are Baryshnikov basis elements, whereas the last tensor factor is non-zero by Theorem \ref{detallado}.
\end{proof}
\end{corollary}

\section{Motion planners for pure planar braids with few strands}\label{seccionhn}
In a recent work~(\cite{Barda}), Bardakov, Singh and Vesnin have proved:
\begin{enumerate}[(i)]
\item\label{unito} $\ppn$ is free of rank $(1,7)$ for $n=(3,4)$;
\item\label{docito} $\ppn$ is not free for $n\geq6$,
\end{enumerate}
and have conjectured: 
\begin{enumerate}[(i)]\addtocounter{enumi}{2}
\item\label{trecito} $\pp_5$ is a free group of rank 31.
\end{enumerate}
The proof of~(\ref{unito}) occupies a full section in~\cite{Barda}. In fact, the authors of that paper offer two different proofs of the freeness of $\pp_4$, one with a geometric flavor and another one with an algebraic flavor. The algebraic proof is technical, whereas the geometric proof is extensive. In this section we give short elementary arguments for both~(\ref{unito}) and (\ref{docito}), as well as a short argument proving a stronger form (Proposition~\ref{weker-bjorner-severs-white} below) of the conjectured~(\ref{trecito}). In addition, we indicate a way to construct an explicit optimal motion planner for $\ppn$ when $n$ is small.

Under this paper's perspective, the simplest case is that of~(\ref{docito}), which is an immediate consequence of Corollary~\ref{tprincipal} and the well-known fact that the topological complexity of a free group is at most 2. Even easier is the case $n=3$ in~(\ref{unito}). Indeed, as observed at the beginning of Section~\ref{secciondetcf}, $\F_k(\mathbb{R},n)$ is either contractible or has the homotopy type of the sphere $S^{k-2}$ for, respectively, $n<k$ or $n=k$. In particular $\pp_1$ and $\pp_2$ are trivial, while (and relevant for~(\ref{unito})) $\pp_3$ is an infinite cyclic group.

Condition~(\ref{trecito}) is a special case of:
\begin{proposition}\label{weker-bjorner-severs-white}
For $3\le k<n<2k$, $\F_k(\mathbb{R},n)$ has the homotopy type of a wedge of $\beta(k,n)$ spheres of dimension $k-2$, where $$\beta(k,n)=\Sigma^n_{i=k} \genfrac(){0pt}{1}{n}{i} \genfrac(){0pt}{1}{i-1}{k-1}.$$
\end{proposition}
\begin{proof}
Severs-White have shown in~\cite[Theorem~1.1]{MR2915650} that $\confkRn$ admits a minimal cellular model, i.e., it has the homotopy type of a cell complex having as many cells in each dimension~$d$ as the rank of the homology (free abelian) group $H^d(\confkRn)$. The result then follows from Theorem~\ref{basbar} and \cite[Theorem~1.1(c)]{MR1317619}.
\end{proof}

We next give an elementary geometric argument leading to a proof of~(\ref{unito}) and~(\ref{trecito}), as well as to a description of explicit motion planners for the corresponding groups $\ppn$. Consider the subspace $X_n\subset\F_3(\mathbb{R},n)$ consisting of the elements $x=(x_1,\dots,x_n)\in\F_3(\mathbb{R},n)$ with $x_n=0$ and $|x|=1$. For instance
\begin{equation}\label{x3}
X_3=\left\{(x_1,x_2,0)\in \mathbb{R}^3\;\colon\;  |(x_1,x_2)|=1 \mbox{ and } (x_1,x_2)\neq(0,0) \right\}=S^1,
\end{equation}
whereas
\begin{align}
X_4&=\left\{(x_1,x_2,x_3,0)\in\mathbb{R}^4\;\colon\; \begin{minipage}{5.4cm}{\footnotesize $|(x_1,x_2,x_3)|=1$,\\ $(0,0)\not\in\{(x_1,x_2),(x_1,x_3),(x_2,x_3)$\} \\ not all $x_1, x_2, x_3$ are equal}\end{minipage}\right\}\nonumber\\
&=S^2-\{\pm(0,0,1),\pm(0,1,0),\pm(1,0,0),\pm\frac1{\sqrt{3}}(1,1,1)\} \nonumber\\
&=\mathbb{R}^2-\left\{\mbox{7 points}\right\}\simeq\bigvee_7 S^1.\label{x4}
\end{align}

\begin{lemma}[{Compare to~\cite[Section~III]{Harsh}}]\label{retraccion}
Let $\mathbb{R}_+$ stand for the positive real numbers. For $n\ge3$, the map $f\colon\!\!\F_3(\mathbb{R},n)\times\mathbb{R}_+\times\mathbb{R}\to\F_3(\mathbb{R},n)$ sending the triple $((x_1,\ldots,x_n),r,a)$ into $(x_1r+a,\ldots,x_nr+a)$ yields, by restriction, a homeomorphism $X_n\times\mathbb{R}_+\times\mathbb{R}\cong\F_3(\mathbb{R},n)$. Consequently, the subspace inclusion $X_n\hookrightarrow\F_3(\mathbb{R},n)$ is a homotopy equivalence.
\end{lemma}
\begin{proof}
For the first assertion, it is straightforward to check that the inverse of the restriction of $f$ to $X_n\times\mathbb{R}_+\times\mathbb{R}$ is given by the map $g\colon\F_3(\mathbb{R},n)\to X_n\times\mathbb{R}_+\times\mathbb{R}$ sending $(x_1,\ldots,x_n)$ into the triple
$$
\left( \frac{1}{N}(x_1-x_n,\ldots,x_{n-1}-x_n,0),N,x_n \right),
$$
where $N$ stands for the norm of $(x_1-x_n,\ldots,x_{n-1}-x_n,0)$. Note that $N\in\mathbb{R}_+$ since $(x_1,\ldots x_n)\in\F_3(\mathbb{R},n)$ and $n\ge3$. For the second assertion, note that the composite $X_n\hookrightarrow\F_3(\mathbb{R},n)\cong X_n\times\mathbb{R}_+\times\mathbb{R}$ takes the form $x\mapsto(x,1,0)$.
\end{proof}

Note that~(\ref{unito}) follows at once from~(\ref{x3}), (\ref{x4}) and Lemma~\ref{retraccion}. A similar proof of~(\ref{trecito}) is also possible; this involves the use of the stereographic projection from the pinched 3-sphere to the 3-dimensional euclidean space in order to express $X_5$ as the complement in $\mathbb{R}^3$ of ten unlinked and untangled curves. Details are omitted.\footnote{In private communications, Harshman and Knapp report having also carried out the reductions for $X_5$ analogous to~(\ref{x3}) and~(\ref{x4}), and which lead to a geometric verification of the fact that $\pp_5$ is free of rank~31.}

Lemma~\ref{retraccion} can be used to construct optimal motion planners on $\F_3(\mathbb{R},n)$ for small values of $n$. Details are based on a couple of reductions using the following standard observation (see~\cite[Theorem~3]{Far}): Assume $\alpha\colon X\to Y$ is a homotopy equivalence with homotopy inverse $\beta\colon Y\to X$. Fix a homotopy $H\colon X\times [0,1]\to X$ between $H_0=\beta\circ\alpha$ and the identity $H_1=Id\colon X\to X$. Assume $s\colon U\to P(Y)$ is a local rule for $Y$ (i.e.~a section for the double evaluation map $e_{0,1}\colon P(Y)\to Y\times Y$) on the open set $U\subseteq Y\times Y$, and set $V=(\alpha\times\alpha)^{-1}(U)$. Then a local rule $\sigma\colon V\to P(X)$ for $X$ is defined through the formula
$$
\sigma(x_1,x_2)(t)=\begin{cases} H(x_1,3t), & \mbox{for $0\le t\le 1/3$;} \\
\beta(s(\alpha(x_1),\alpha(x_2))(3t-1)), & \mbox{for $1/3\leq t\leq 2/3$;} \\
H(x_2,3(1-t)), & \mbox{for $2/3\leq t\le1$.} \end{cases}
$$

Applying the construction above to the homotopy equivalence $$F=\left(\F_3(\mathbb{R},n)\stackrel{\cong\;}\to X_n\times\mathbb{R}_+\times\mathbb{R}\stackrel{\mbox{\tiny proj}\;\,}{\longrightarrow}X_n\right),$$ we see that it suffices to describe an optimal motion planner on $X_n$. (Explicit formulae for $F$, the needed homotopy inverse $G$, and the needed homotopy between the identity and the composite $G\circ F$ are easily deduced from the proof of Lemma~\ref{retraccion}.) In turn, since $X_n$ has the homotopy type of a wedge of circles (for $n\le5$), and since explicit optimal motion planners for finite wedges of spheres have been described in~\cite{CohenPruid08} (see also~\cite{GGY}), it suffices to describe explicit homotopy equivalences (going in both directions) $X_3\simeq S^1$, $X_4\simeq\vee_7S^1$ and $X_5\simeq\vee_{31}S^1$. The latter task has been accomplished in~(\ref{x3}) and~(\ref{x4}) for $n=3$ and $n=4$, where an obvious stereographic projection is needed in the latter case. The resulting motion planner in $\F_3(\mathbb{R},3)$ is spelled out next.

\begin{example}{\em
Let $D_0$ be the subspace of $\F_3(\mathbb{R},3)\times\F_3(\mathbb{R},3)$ consisting of pairs $(x,y)$ such that the line segment $[x,y]$ in $\mathbb{R}^3$ from $x$ to $y$ does not intersect the diagonal $\Delta=\{(z,z,z)\colon z\in\mathbb{R}\}$, and let $D_1$ be the complement of $D_0$ in $\F_3(\mathbb{R},3)\times\F_3(\mathbb{R},3)$. Both $D_0$ and $D_1$ are ENR's, so it suffices to describe a local rule on each. Motion planning in $\F_3(\mathbb{R},3)$ for points $(x,y)\in D_0$ can be done by following the segment $[x,y]$. On the other hand, for $(x,y)\in D_1$, let $p(x,y)$ be the point where the segment $[x,y]$ intersects $\Delta$. Since the vectors $y-x$ and $(1,1,1)$ are linearly independent, their cross product $u(x,y)$ is nonzero. We then motion plan in $\F_3(\mathbb{R},3)$ from $x$ to $y$ (with $(x,y)\in D_1$) by following first the segment $[x,p(x,y)+u(x,y)]$, and then the segment $[p(x,y)+u(x,y),y]$. 
}\end{example}

\section{Configuration spaces with controlled collisions}\label{sectioncontrolled}

No-$k$-equal configuration \red{spaces are special instances of configuration spaces} with collisions \red{controlled} by a simplicial complex. Fix a space $X$, a positive integer $n$, and a simplicial complex $K$ with vertex set $\{1,\ldots,n\}$. Let $\Delta^{n-1,d}$ stand for the $d$-dimensional skeleton of $\Delta^{n-1}$, the simplex spanned by $\{1,\ldots,n\}$.

For a subset $\sigma\subseteq\{1,\ldots,n\}$ consider the partial diagonal subspace
$$
D_\sigma:=\{(x_1,\ldots, x_n)\in X^n\colon \card(\{x_i\colon i\in\sigma\})=1\}.
$$

\begin{definition}\label{confcontrolled}
The $K$-diagonal in $X^n$ is $D_K=\bigcup_{\sigma}D_\sigma,$ where $\sigma$ runs over the non-faces of $K$. We set
$$
\F_K(X,n)=X^n - D_K,
$$
which we call the configuration space of $n$ points in $X$ with collisions controlled by $K$.
\end{definition}

By definition, for $(x_1,\ldots,x_n)\in\F_K(X,n)$, a multiple collision $x_{i_0}=x_{i_1}=\cdots=x_{i_d}$ can hold only if $\{i_0,i_1,\ldots,i_d\}\in K$.

Note that $D_\sigma\subseteq D_\tau$ provided $\tau\subseteq\sigma$. Consequently:
\begin{enumerate}
\item $\F_L(X,n)\subseteq\F_K(X,n)$, provided $L$ is a subcomplex of $K$.
\item $\F_K(X,n)=X^n-\bigcup_\sigma D_\sigma$, where $\sigma$ runs over the \emph{minimal} non-faces of $K$.
\end{enumerate}
For instance:
\begin{enumerate}\addtocounter{enumi}{2}
\item $\F_{\Delta^{n-1}}(X,n)=X^n$.
\item $\F_{\Delta^{n-1,0}}(X,n)=\F(X,n)$, the usual configuration space.
\item $\F_{\Delta^{n-1,k-2}}(X,n)=\F_k(X,n)$, the no-$k$-equal configuration space of $X$.
\end{enumerate}

\red{Although} configuration spaces with collisions controlled by simplicial complexes \red{have appear previously in the literature, most of their algebraic topology properties are presently unknown. An interesting challenge for future research is to compute their mod 2 cohomology algebras, say as modules over the Steenrod algebra, in terms of the combinatorial properties of the controlling simplicial complexes. As in the case of $\F_k(\mathbb{R},n)$, this might give enough information to compute the topological complexity of $\F_K(\mathbb{R},n)$ as a function of $n$ and $K$.}


\medskip
{\small \sc Departamento de Matem\'aticas

Centro de Investigaci\'on y de Estudios Avanzados del I.P.N.

Av.~Instituto Polit\'ecnico Nacional n\'umero 2508

San Pedro Zacatenco, M\'exico City 07000, M\'exico

{\tt jesus@math.cinvestav.mx}

{\tt jlleon@math.cinvestav.mx}}

\bigskip\smallskip
{\small \sc Instituto de Matem\'aticas

Universidad Nacional Aut\'onoma de M\'exico

Le\'on No.2, Altos, Oaxaca de Ju\'arez 68000, M\'exico

{\tt croque@im.unam.mx}}

\end{document}